\documentclass[11pt, a4paper,reqno]{amsart}

\newcommand{\reg}{regular}
\newcommand{\rqr}{right quasi-regular}

\usepackage{amssymb}

\usepackage{enumitem}
\setlist[enumerate,1]{label=\rm(\arabic*)}
\setlist[enumerate,2]{label=\rm(\alph*)}
\setlist[enumerate,3]{label=\rm(\roman*)}

\usepackage[PS]{diagrams}
\usepackage{xcolor}
\usepackage{pgf}

\newcommand{\ie}{\textit{i.e.}}

\newcommand{\ignore}[1]{\relax}

\newcommand{\bZ}{\mathbb{Z}}

\numberwithin{equation}{section}

\newtheorem{theorem}[equation]{Theorem}
\newtheorem{corollary}[equation]{Corollary}

\newtheorem{lemma}[equation]{Lemma}

\theoremstyle{definition}
\newtheorem{definition}[equation]{Definition}

\newtheorem{example}[equation]{Example}

\let\ideal=\unlhd

\usepackage[hypertexnames=false,colorlinks=false,pdfborderstyle={/S/U/W 1}]{hyperref}
\usepackage{bookmark}

\begin{document}

\title%
[A characterisation of the Jacobson radical]%
{A ``quite superfluous'' characterisation of the Jacobson radical of an associative ring}

\date{\today}

\author{Thomas H\"uttemann}

\address{Thomas H\"uttemann\\ Queen's University Belfast\\ School of
  Mathematics and Physics\\ Mathematical Sciences Research Centre\\
  Belfast BT7~1NN\\ Northern Ireland, UK}

\email{t.huettemann@qub.ac.uk}

\urladdr{http://www.qub.ac.uk/puremaths/Staff/Thomas\ Huettemann/}


\begin{abstract}
  It is well-known that the \textsc{Jacobson} radical of a unital
  ring~$R$ is the largest superfluous right ideal of~$R$. It is
  recorded here that the result carries over to non-unital rings
  provided the notion of ``superfluous'' is taken relative to all {\it
    regular\/} ideals.
\end{abstract}

\vglue -1cm

\maketitle

\section{The Jacobson radical of a ring}
\label{sec:J(R)}

Let $R$ be an associative ring, possibly non-unital. An element
$a \in R$ is said to be {\it\rqr{}\/} if there exists $x \in R$ with
$ax = a+x$. (In case $R$ is a unital ring, this is equivalent to $1-a$
having right inverse $1-x$.) The {\it \textsc{Jacobson} radical~$J(R)$
  of~$R$\/} is the two-sided ideal
\begin{displaymath}
  J(R) = \{ a \in R \,|\, \forall b \in R: %
  \text{\(ab\) is \rqr} \} \ ,
\end{displaymath}
see, for example, \cite[Definition~6.6]{MR0188241}.

\begin{example}
  \label{ex:R_0}
  Write $Q$ for the (commutative) unital ring of rational numbers of
  the form $s/t$ for $s,t \in \bZ$ with $2 \nmid t$ and $3 \nmid
  t$.
  The units in~$Q$ are precisely the elements represented by fractions
  $s/t$ with $2,3 \nmid t$ and $2,3 \nmid s$, in which case the
  inverse of~$s/t$ is~$t/s$.

  Let $R_{0} = (2)$, the (right) ideal consisting of fractions $2s/t$
  with $2,3 \nmid t$, considered as a non-unital ring. We claim that
  {\it the \textsc{Jacobson} radical of~$R_{0}$ is the ideal $J = (6)$
    of fractions of the form $6s/t$ with $2,3 \nmid t$}.

  As $R_{0} \ideal Q$ we have $J(R_{0}) = R_{0} \cap J(Q)$,
  by \cite[Theorem~6.14]{MR0188241}. Hence it is enough to verify that
  $J(Q) = J$.
  To show $J(Q) \supseteq J$ let $a = 6s/t \in J$ and
  $b = s'/t' \in Q$, where $2,3 \nmid t,t'$. As $tt'-6ss'$ is not
  divisible by either~$2$ or~$3$, the quantity
  \begin{displaymath}
    1-ab = 1 - \frac{6s}{t} \cdot \frac{s'}{t'} = \frac{tt' -
      6ss'}{tt'}
  \end{displaymath}
  is a unit in~$Q$ whence $ab$ is~\rqr. It follows that $a \in J(Q)$.
  To show $J(Q) \subseteq J$ suppose that $a = s/t \notin J$, where
  $2,3 \nmid t$. Then $2 \nmid s$ or~$3 \nmid s$. If $3 \nmid s$
  choose $b \in \bZ$ such that $sb \equiv t \mod~3$. Then $3|(t-sb)$,
  the quantity
  \begin{displaymath}
    1-ab = \frac{t-sb}{t}
  \end{displaymath}
  is not a unit in~$Q$, and $ab$ is not~\rqr. It follows that
  $a \notin J(R_{0})$ as required.---The case $2 \nmid s$ is dealt
  with by a similar argument. \qed
\end{example}

\section{Superfluous ideals and the Jacobson radical}
\label{sec:unital}

A right ideal $J \ideal R$ is called {\it superfluous\/} if for every
right ideal $I \ideal R$,
\begin{displaymath}
  J+I = R \quad \Longrightarrow \quad I = R \ .
\end{displaymath}
The following well-known result can be found 
in \textsc{Anderson}-\textsc{Fuller} \cite[Theorem~15.3]{MR1245487}.

\begin{theorem}
  \label{thm:superfluous}
  Let $R$ be a unital ring. Its \textsc{Jacobson} radical~$J(R)$ is a
  superfluous right ideal of~$R$ which contains every superfluous
  right ideal of~$R$ (\ie, $J(R)$ is the unique largest superfluous
  right ideal of~$R$). \qed
\end{theorem}

In the next section we will generalise this result to include the case
of non-unital rings. We note first that {\it $J(R)$ need not be
  superfluous in general\/}:

\begin{example}
  \label{ex:not_superfluous}
  Consider the ring $R_{0} = (2) \ideal Q$ from Example~\ref{ex:R_0},
  and let $I = (4) \lhd R_{0}$ be the proper (right) ideal of~$R_{0}$
  consisting of fractions $4s/t$ with $2,3 \nmid t$.  Then
  $J(R_{0}) + I = R_{0}$; for given $a = 2s/t \in R_{0}$ with
  $2,3 \nmid t$ we have trivially
  \begin{displaymath}
    a = \frac{2s}{t} = \frac{6s}{t} - \frac{4s}{t} \in (6) + (4) =
    J(R_{0}) + I \ .
  \end{displaymath}
  As $I \subsetneqq R_{0}$ this shows that $J(R_{0})$ is not a
  superfluous right ideal. \qed
\end{example}

\section{quite superfluous ideals and the Jacobson radical}
\label{sec:non-unital}

Recall that a right ideal $I$ of a ring~$R$ is called {\it \reg\/}
if there exists an element $e \in R$ such that $er-r \in I$ for all
$r \in R$. We call $e$ a {\it regulator\/} of~$I$.

\begin{definition}
  A right ideal $J$ of~$R$ is called {\it quite superfluous\/} if $J+I=R$ implies
  $I=R$ for every \reg{} right ideal~$I$ of~$R$.
\end{definition}

Every superfluous right ideal is quite superfluous. Every right ideal contained
in a quite superfluous right ideal is quite superfluous.

\begin{lemma}[{compare \textsc{McCoy} \cite[Lemma~6.18]{MR0188241}}]
  \label{lem:max}
  Let $I$ be a regular right ideal of~$R$ with regulator~$e$. Suppose
  that $I \neq R$. Then $e \notin I$, the element $e$ is not \rqr, and
  there exists a \reg{} maximal right ideal $M$ containing~$I$ but
  not~$e$.
\end{lemma}

\begin{proof}
  If $e \in I$ were true we had $er \in I$ and hence
  $r = er-(er-r) \in I$ for all~$r$ whence $I=R$, contrary to our
  hypothesis that $I$~is a proper right ideal. Hence $e \notin I$.
  Moreover, $e$ is not \rqr. For otherwise there existed $x \in R$
  such that $ex = e+x$ whence $e = ex-x \in I$.

  By \textsc{Zorn}'s lemma there exists a right ideal $M \supseteq I$
  which is maximal among all proper right ideals containing~$I$ but
  not~$e$. Now $M$ is \reg{} as $I$~is, with the same
  regulator~$e$. Moreover, $M$ is a maximal right ideal. For if
  $N \supsetneqq M$ then $e$ is a regulator of~$N$ as well, and we
  must have $e \in N$ by maximality of~$M$. This implies that
  $x = ex-(ex-x) \in N$ for all $x \in R$ whence $N = R$.
\end{proof}

It is known that $J(R)$ is the intersection of all regular maximal
right ideals of~$R$ \cite[Theorem~6.20]{MR0188241}. Hence $J(R) = R$
if and only if $R$ has no regular right ideals distinct from~$R$, in
which case every right ideal of~$R$ is quite superfluous.

\begin{theorem}
  \label{thm:weakly_superfluous}
  Let $R$ be an associative ring. Its \textsc{Jacobson} radical~$J(R)$
  is a quite superfluous right ideal containing every quite superfluous right ideal of~$R$
  (\ie, $J(R)$ is the unique largest quite superfluous right ideal of~$R$).
\end{theorem}

If $R$ is unital then every right ideal is \reg{} with regulator
$e=1$, and ``quite superfluous'' is the same as ``superfluous''. Thus
Theorem~\ref{thm:weakly_superfluous} indeed generalises
Theorem~\ref{thm:superfluous}.

\begin{proof}[Proof of Theorem~\ref{thm:weakly_superfluous}]
  In view of the remarks preceding the Theorem, we may assume that
  $J(R) \neq R$. We start by proving that $J(R)$ is quite superfluous. Suppose that
  $I \subsetneqq R$ is a \reg{} right ideal of~$R$; we will show that
  $J(R)+I \neq R$.  Since $I$ is \reg, we can choose a regulator
  $e \in R$ so that $er-r$ lies in~$I$ for all~$r$. By
  Lemma~\ref{lem:max} there exists a regular maximal right ideal $M$
  containing~$I$. As the \textsc{Jacobson} radical is contained in
  every \reg{} maximal ideal, we conclude
  $J(R) + I \subseteq J(R) + M = M \neq R$.

  Next, we prove that $J(R)$ contains every quite superfluous right ideal
  of~$R$. Suppose that the right ideal~$K$ of~$R$ is not contained
  in~$J(R)$. Then there exist elements $a \in K$ and $b \in R$ such
  that $ab$ is not \rqr. Then $ab$ cannot be contained in the \reg{}
  right ideal $I = \{abr-r \,|\, r \in R \}$. By Lemma~\ref{lem:max},
  applied to $e =ab$, there exists a \reg{} maximal right ideal~$M$
  of~$R$ containing~$I$ with $ab \notin M$. As $ab \in K$ we have
  $K + M \supsetneqq M$ whence $K + M = R$ as $M$ is a maximal
  ideal. But $M \neq R$, so $K$ is not quite superfluous.
\end{proof}

\begin{corollary}
  A right ideal $I$ of~$R$ is quite superfluous if and only if it is contained
  in~$J(R)$. \qed
\end{corollary}

Thus the \textsc{Jacobson} radical of~$R$ can be described as
\begin{equation}
  \label{eq:1} 
  J(R) = \{ a \in R \,|\, \text{\((a)\) is quite superfluous} \} \ ,
\end{equation}
where $(a)$ denotes the smallest right ideal of~$R$ containing~$a$. As
the \textsc{Jacobson} radical of the ring $R_{0}$ of
Example~\ref{ex:R_0} contains the element $a=6$, but
$(a) = J(R_{0})$ is not a superfluous (right) ideal by
Example~\ref{ex:not_superfluous}, we cannot replace ``quite superfluous'' by
``superfluous'' in~\eqref{eq:1}.

\medbreak

A characterisation of the \textsc{Jacobson} radical closely related
to~\eqref{eq:1} has been given by \textsc{Kert\'esz} \cite{MR0151488}:
\begin{equation} 
  \label{eq:2}
  J(R) = \{ a \in R \,|\, \forall s \in R \colon \text{\((sa)\) is
    superfluous} \} \ . 
\end{equation}
By comparing \eqref{eq:1} with~\eqref{eq:2} we see that the right
ideal $(a)$ is quite superfluous if and only if $(sa)$ is
superfluous for all $s \in R$.

\raggedright


\begin{thebibliography}{McC64}

\bibitem[AF92]{MR1245487}
Frank~W. Anderson and Kent~R. Fuller.
\newblock {\em Rings and categories of modules}, volume~13 of {\em Graduate
  Texts in Mathematics}.
\newblock Springer-Verlag, New York, second edition, 1992.

\bibitem[Ker63]{MR0151488}
A.~Kert\'esz.
\newblock A characterization of the {J}acobson radical.
\newblock {\em Proc. Amer. Math. Soc.}, 14:595--597, 1963.

\bibitem[McC64]{MR0188241}
Neal~H. McCoy.
\newblock {\em The theory of rings}.
\newblock The Macmillan Co., New York; Collier-Macmillan Ltd., London, 1964.

\end{thebibliography}
\end{document}